\documentclass[aos,preprint]{imsart}

\RequirePackage[OT1]{fontenc}
\RequirePackage{amsthm,amsmath}
\RequirePackage[numbers]{natbib}
\RequirePackage[colorlinks,citecolor=blue,urlcolor=blue]{hyperref}

\arxiv{arXiv:0000.0000}

\startlocaldefs
\numberwithin{equation}{section}
\theoremstyle{plain}
\newtheorem{thm}{Theorem}[section]
\endlocaldefs

\newtheorem{lemma}{Lemma}[section]

\begin{document}

\begin{frontmatter}
\title{Estimation of the Weighted Integrated Square Error of the Grenander Estimator by the Kolmogorov--Smirnov Statistic}
\runtitle{Integrated Square Error of the Grenander Estimator}

\begin{aug}
\author{\fnms{Malkhaz} \snm{Shashiahsvili}\ead[label=e1]{malkhaz.shashiashvili@tsu.ge}}

\runauthor{M. Shashiashvili}

\affiliation{Ivane Javakhishvili Tbilisi State University}

\address{Ivane Javakhishvili Tbilisi State University \\ Faculty of Exact and Natural Sciences \\ Department of Mathematics \\ 13 University St. Tbilisi 0186, Georgia \\
\printead{e1}\\ }

\end{aug}

\begin{abstract}
We consider in this paper the Grenander estimator of unbounded, in general, nonincreasing densities on the interval $[0,1]$ without any smoothness assumptions. For fixed number $n$ of i.i.d. random variables $X_1,X_2,\dots,X_n$ with values in $[0,1]$ and the nonincreasing density function $f(x)$, $0\leq x\leq 1$, we prove an inequality bounding the weighted integrated square error of the Grenander estimator with probability one by the classical Kolmogorov--Smirnov statistic. Further, we consider some interesting implications of the latter inequality
\end{abstract}



\end{frontmatter}

\section{Introduction}
\label{sec:1}

Nonparametric density estimation has mainly been devoted for a long time, to estimation of smooth densities using linear methods like kernel estimators with fixed bandwidth or projection estimators (truncated series expansions with estimated coefficients). Suppose we know only that $f(x)$, $0\leq x\leq 1$, is a nonincreasing density (unbounded and discontinuous, in general). It can very well be steep at some places and flat elsewhere Consequently, the best histogram for estimating $f(x)$, $0\leq x\leq 1$, need not be based on a regular partition as can be seen from the results of Birg\'{e} (1987, a, b). In this case a special estimator, which takes the form of a variable binwidth histogram has been known for a long time. It has been introduced by Grenander (1956) as the left derivative of the least concave majorant of the empirical distribution function, which is merely the nonparametric maximum likelihood estimator restricted to nonincreasing densities on $[0,1]$ (for a proof see, for instance, Grenander (1981)). The so-called Grenander estimator has been studied by Prakasa Rao (1969), Kiefer and Wolfowitz (1976), Grenander (1981), Groeneboom (1985) and Birg\'{e} (1989).

Let $X_1,X_2,\dots,X_n$ be i.i.d. random variables with values in $[0,1]$ and the nonincreasing density function $f(x)$, $0\leq x\leq 1$, and let $F(x)$, $0\leq x\leq 1$, be the corresponding cumulative distribution function with $F(0)=0$ and
\begin{equation}\label{eq:1.1}
    F(x)=\int\limits_0^x f(y)\,dy, \;\; 0\leq x\leq 1, \quad F(1)=\int\limits_0^1 f(y)\,dy=1.
\end{equation}

We shall assume that $f(x)$, $0\leq x\leq 1$, is a right-continuous version of the density function. Then it is evident that $f(x)$, $0\leq x\leq 1$, is a right derivative of the concave function $F(x)$ and its left limit $f(x-)$, $0<x\leq 1$, coincides with the left derivative of $F(x)$, $0\leq x\leq 1$. It is a well-known mathematical fact that $f(x)=f(x-)$ everywhere except a countable set of points $x$, $0\leq x\leq 1$. Consider now another right-continuous nonincreasing density function $\varphi(x)$, $0\leq x\leq 1$. We wish to introduce the integrated square distance
\begin{equation}\label{eq:1.2}
    \int\limits_0^1 (\varphi(x)-f(x))^2\,dx
\end{equation}
between the density functions $\varphi(x)$ and $f(x)$, but as we consider, in general, unbounded densities, satisfying
\begin{equation}\label{eq:1.3}
    \lim_{x\downarrow 0} f(x)=+\infty, \quad \lim_{x\downarrow 0} \varphi(x)=+\infty
\end{equation}
the expression \eqref{eq:1.2} becomes inconsistent.

Indeed if we consider the nonincreasing (and unbounded at $0)$ densities
\begin{equation}\label{eq:1.4}
    f_{\alpha}(x)=(1-\alpha)\cdot x^{-\alpha}, \;\; 0\leq x\leq 1, \;\;\text{where}\;\; \frac{1}{2}<\alpha<1,
\end{equation}
we can simply check, that
\begin{equation}\label{eq:1.5}
    \int\limits_0^1 f_{\alpha}^2(x)\,dx=+\infty.
\end{equation}

It turns out that for nonincreasing densities the following weighted integrated square distance
\begin{equation}\label{eq:1.6}
    \int\limits_0^1 (\varphi(x)-f(x))^2x\,dx
\end{equation}
has a full sense and it will be considered in Section \ref{sec:2}.

Let $F_n(x)$, $0\leq x\leq 1$, be the empirical distribution function, constructed from i.i.d. random variables $X_1,X_2,\dots,X_n$, having the nonincreasing (and unbounded at $0$, in general), density function $f(x)$, $0\leq x\leq 1$. As the corresponding \eqref{eq:1.1} distribution function $F(x)$ is absolutely continuous with $F(0)=0$, with probability one we have that the random variable
\begin{equation}\label{eq:1.7}
    X_{\min}=\min(X_1,X_2,\dots,X_n)
\end{equation}
is strictly positive and hence
\begin{equation}\label{eq:1.8}
    F_n(x)=0 \;\;\text{if}\;\; 0\leq x<X_{\min}\;\; \text{$(P$-a.s.$)$}.
\end{equation}
Write $\widehat{F}_n(x)$, $0\leq x\leq 1$, for the least concave majorant of $F_n(x)$, $0\leq x\leq 1$, and let $\widehat{f}_n(x)$ denote the right derivative of the latter concave majorant. We get from \eqref{eq:1.8} that
\begin{equation}\label{eq:1.9}
    \widehat{F}_n(0)=0 \;\; \text{$(P$-a.s.$)$}
\end{equation}
and also
\begin{equation}\label{eq:1.10}
    \widehat{f}_n(x)=constant, \;\;\text{if}\;\; 0\leq x<X_{\min} \;\; \text{$(P$-a.s.$)$}.
\end{equation}
We have evidently
\begin{equation}\label{eq:1.11}
\begin{gathered}
    \widehat{F}_n(x)=\int\limits_0^x \widehat{f}_n(y)\,dy, \;\; 0\leq x\leq 1, \\
    \widehat{F}_n(1)=\int\limits_0^1 \widehat{f}_n(y)\,dy=1
\end{gathered} \quad \text{$(P$-a.s.$)$}.
\end{equation}

As $F_n(x)$, $0\leq x\leq 1$, is the nondecreasing function. the same property holds for its least concave majorant $\widehat{F}_n(x)$ and hence $\widehat{f}_n(x)$, $0\leq x\leq 1$, is a nonnegative function. Moreover, the latter function is nonincreasing as the right derivative of the concave function. By its construction the function $\widehat{F}_n(x)$, $0\leq x\leq 1$, is piecewise linear concave function and as a result we get that the function $\widehat{f}_n(x)$, $0\leq x\leq 1$, is right-continuous step function, nonnegative and nonincreasing. The left limit $\widehat{f}_n(x-)$, $0<x\leq 1$, of the function $\widehat{f}_n(x)$, $0\leq x\leq 1$, coincides with the left derivative of the least concave majorant $\widehat{F}_n(x)$, $0\leq x\leq 1$, and hence it is the celebrated Grenander estimator of the unknown nonincreasing density function $f(x)$, $0\leq x\leq 1$.

In Section \ref{sec:2} we shall establish our main result (Theorem \ref{th:2.1}) which states that the following weighted integrated square error of the Grenander estimator
\begin{equation}\label{eq:1.12}
    \int\limits_0^1 (\widehat{f}_n(x-)-f(x-))^2x\,dx
\end{equation}
is bounded with probability one by the classical Kolmogorov--Smirnov statistic
\begin{equation}\label{eq:1.13}
    2\sup_{0\leq x\leq 1} |F_n(x)-F(x)|.
\end{equation}
From Theorem \ref{th:2.1} we shall deduce several interesting consequences:

First of all, the Grenander estimator $\widehat{f}_n(x-)$, $0<x\leq 1$, is the consistent estimate of the unknown nonincreasing density $f(x)$, $0\leq x\leq 1$, in the sense of the weighted integrated square distance \eqref{eq:1.6}, next, for arbitrary probability close to $1$, we can find such a number $n$ of observations $X_1,X_2,\dots,X_n$, that the weighted integrated square error \eqref{eq:1.12} will be small enough with the prescribed high probability, and the third consequence gives the bound of the quadratic risk of the Grenander estimator
\begin{equation}\label{eq:1.14}
    E\int\limits_0^1 (\widehat{f}_n(x-)-f(x))^2x\,dx\leq \sqrt{2\pi}\,n^{-1/2}.
\end{equation}

\section{The formulation and the proof of the preliminary lemmas and the main result}
\label{sec:2}

We start this section with two simple lemmas which are needed to establish the basic result of this paper.

\begin{lemma}\label{lem:2.1}
We have for arbitrary nonincreasing density $f(x)$, $0\!\leq\!x\!\leq~\!\!1$,
\begin{equation}\label{eq:2.1}
    0\leq f(x)\cdot x\leq F(x), \quad \lim_{x\downarrow 0} f(x)\cdot x=0.
\end{equation}
\end{lemma}

\begin{proof}
From the nonincreasing property of the function $f(x)$, $0\leq x\leq 1$, we get
\begin{equation}\label{eq:2.2}
    F(x)=\int\limits_0^x f(y)\,dy\geq \int\limits_0^x f(x)\,dy=f(x)\cdot x, \;\; x\geq 0, \;\; 0\leq x\leq 1.
\end{equation}
As $\lim\limits_{x\downarrow0} F(x)=F(0)=0$, we come to relations \eqref{eq:2.1}.
\end{proof}

From Lemma \ref{lem:2.1} we obtain
\begin{equation}\label{eq:2.3}
     0\leq \widehat{f}_n(x)\cdot x\leq \widehat{F}_n(x), \quad \lim_{x\downarrow 0} \widehat{f}_n(x)\cdot x=0.
\end{equation}

\begin{lemma}\label{lem:2.2}
We have the following bound for arbitrary nonincreasing density $f(x)$, $0\leq x\leq 1$,
\begin{equation}\label{eq:2.4}
    \int\limits_0^1 f^2(x)\cdot x\,dx\leq\frac{1}{2}\,.
\end{equation}
\end{lemma}

\begin{proof}
We use the the inequality \eqref{eq:2.1} and write
$$      \int\limits_0^1 f(x)(f(x)\cdot x)\,dx\leq \int\limits_0^1 f(x)\cdot F(x)\,dx=\int\limits_0^1 \frac{1}{2}\,dF^2(x)=\frac{1}{2}\,.   \qedhere     $$
\end{proof}

From the latter lemma we can write
\begin{equation}\label{eq:2.5}
    \int\limits_0^1 \widehat{f}_n{}^{2}(x)\cdot x\,dx\leq\frac{1}{2}\,, \quad \int\limits_0^1 (\widehat{f}_n(x)-f(x))^2x\,dx\leq 1.
\end{equation}
Introduce now the notations
\begin{equation}\label{eq:2.6}
    \widehat{G}_n(x)=\widehat{F}_n(x)-F(x), \quad \widehat{g}_n(x)=\widehat{f}_n(x)-f(x), \;\; 0\leq x\leq 1,
\end{equation}
and note that
\begin{equation}\label{eq:2.7}
    \widehat{G}_n(x)=\int\limits_0^x \widehat{g}_n(y)\,dy=\int\limits_0^x \widehat{g}_n(y-)\,dy, \;\; 0\leq x\leq 1,
\end{equation}
where
\begin{equation}\label{eq:2.8}
    \widehat{g}_n(x-)=\widehat{f}_n(x-)-f(x-), \;\; 0<x\leq 1,
\end{equation}
and note also the obvious equalities
\begin{multline}\label{eq:2.9}
    \int\limits_0^1 (\widehat{f}_n(x-)-f(x))^2x\,dx \\
    =\int\limits_0^1 (\widehat{f}_n(x-)-f(x-))^2x\,dx=\int\limits_0^1 (\widehat{f}_n(x)-f(x))^2x\,dx.
\end{multline}

We will establish the following non-asymptotic result valid with probability one for any number $n$ of i.i.d. observations $X_1,X_2,\dots,X_n$, having nonincreasing (and unbounded, in general) unknown probability density function $f(x)$, $0\leq x\leq 1$,

\begin{thm}[Main result]\label{th:2.1}
The weighted integrated square error of the Grenander estimator $\widehat{f}_n(x-)$, $0<x\leq 1$, is bounded with probability one by the Kolmogorov--Smirnov statistic, that is
\begin{equation}\label{eq:2.10}
    \int\limits_0^1 (\widehat{f}_n(x-)-f(x))^2x\,dx\leq 2\sup_{0\leq x\leq 1} |F_n(x)-F(x)| \;\;\text{$(P$-a.s.$)$}.
\end{equation}
\end{thm}

\begin{proof}
The left-hand side of the above inequality is bounded by $1$ according to Lemma \ref{lem:2.2} (see the inequality \eqref{eq:2.5}). Take $\delta>0$ arbitrary small, $0<\delta<1$. The functions $\widehat{G}_n(x)$ and $\widehat{g}_n(x)$, $0\leq x\leq 1$, are right-continuous functions of bounded variation on the interval $[\delta,1]$ (we remind that $\lim\limits_{x\downarrow 0}f(x)$ can be equal to $+\infty)$.

The important formula of the integration by parts is valid for the functions of bounded variation $\widehat{G}_n(x)$ and $\widehat{g}_n(x)$ on the interval $[\delta,1]$ and has the following form (see Hewitt, Stromberg (1975), Th. 21.67)
\begin{equation}\label{eq:2.11}
\begin{gathered}
    d(\widehat{G}_n\cdot\widehat{g}_n)=\widehat{G}_n\cdot d\widehat{g}_n+\widehat{g}_n(-)\cdot d\widehat{G}_n, \;\;\text{or}\;\; \\
            \widehat{g}_n(-)\cdot d\widehat{G}_n=d(\widehat{G}_n\cdot \widehat{g}_n)-\widehat{G}_n\cdot d\widehat{g}_n,
\end{gathered}
\end{equation}
which after multiplication by $x$, and the subsequent integration, becomes
\begin{equation}\label{eq:2.12}
    \int\limits_{\delta}^1 x\cdot \widehat{g}_n(-)\,d\widehat{G}_n=\int\limits_{\delta}^1 x\,d(\widehat{G}_n\cdot \widehat{g}_n)-\int\limits_{\delta}^1 x\cdot \widehat{G}_n\,d\widehat{g}_n.
\end{equation}
We have
\begin{gather}
    \int\limits_{\delta}^1 x\cdot \widehat{g}_n(x-)\,d\widehat{G}_n(x)=\int\limits_{\delta}^1 (\widehat{g}_n(x-))^2x\,dx, \label{eq:2.13} \\
    \int\limits_{\delta}^1 x\,d(\widehat{G}_n(x)\cdot\widehat{g}_n(x))=x\cdot\widehat{G}_n(x)\cdot\widehat{g}_n(x)\Big|_{\delta}^1-\int\limits_{\delta}^1 \widehat{G}_n(x)\cdot\widehat{g}_n(x)\,dx \nonumber \\
    =-\delta\widehat{G}_n(\delta)\cdot\widehat{g}_n(\delta)-\int\limits_{\delta}^1 \frac{1}{2}\,d(\widehat{G}_n(x))^2=
                -\delta\cdot\widehat{G}_n(\delta)\cdot\widehat{g}_n(\delta)+\frac{1}{2}\,(\widehat{G}_n(\delta))^2, \nonumber
\end{gather}
as $\widehat{G}_n(1)=0$. Thus we get
\begin{equation}\label{eq:2.14}
    \int\limits_{\delta}^1 (\widehat{g}_n(x-))^2x\,dx=-\delta\widehat{G}_n(\delta)\cdot\widehat{g}_n(\delta)+\frac{1}{2}(\widehat{G}_n(\delta))^2-\int\limits_{\delta}^1 x\cdot \widehat{G}_n(x)\,d\widehat{g}_n(x).
\end{equation}

Let us bound the last term of the latter equality \eqref{eq:2.14}
\begin{multline}
    \bigg|-\int\limits_{\delta}^1 x\cdot\widehat{G}_n(x)\,d\widehat{g}_n(x)\bigg|\leq \int\limits_{\delta}^1 x\cdot|\widehat{G}_n(x)|\,d({\rm var}\,\widehat{g}_n)(x) \\
    \leq \sup_{0\leq x\leq 1} |\widehat{G}_n(x)|\int\limits_{\delta}^1 x\,d\big((-\widehat{f}_n(x))+(-f(x))\big). \label{eq:2.15}
\end{multline}

We have
\begin{multline*}
    \int\limits_{\delta}^1 x\,d\big((-\widehat{f}_n(x))+(-f(x))\big) \\
    =x(-\widehat{f}_n(x))+(-f(x))\Big|_{\delta}^1+\int\limits_{\delta}^1 (\widehat{f}_n(x)+f(x))\,dx \\
    \leq \delta(\widehat{f}_n(\delta)+f(\delta))+\int\limits_{\delta}^1 (\widehat{f}_n(x)+f(x))\,dx
\end{multline*}
as $-\widehat{f}_n(1)-f(1)\leq 0$, hence we get the bound
\begin{multline}
    \bigg|-\int\limits_{\delta}^1 x\cdot\widehat{G}_n(x)\,d\widehat{g}_n(x)\bigg| \\
    \leq \sup_{0\leq x\leq 1} |\widehat{G}_n(x)|\,\bigg[\delta(\widehat{f}_n(\delta)+f(\delta))+\int\limits_{\delta}^1 (\widehat{f}_n(x)+f(x))\,dx\bigg]. \label{eq:2.16}
\end{multline}
From the equality \eqref{eq:2.14} and the bound \eqref{eq:2.16} we come to the inequality
\begin{multline}
    \int\limits_{\delta}^1 (\widehat{g}_n(x-))^2x\,dx\leq\delta(\widehat{f}_n(\delta)+f(\delta))|\widehat{G}_n(\delta)|+\frac{1}{2}\,(\widehat{G}_n(\delta))^2 \\
    +\sup_{0\leq x\leq 1} |\widehat{F}_n(x)-F(x)|\bigg[\delta(\widehat{f}_n(\delta)+f(\delta))+\int\limits_{\delta}^1 (\widehat{f}_n(x)+f(x))\,dx\bigg]. \label{eq:2.17}
\end{multline}

Taking into account the Lemmas \ref{lem:2.1} and \ref{lem:2.2}, in particular the relations \eqref{eq:2.1}, \eqref{eq:2.3} and \eqref{eq:2.5} and tending the small parameter $\delta$ to $0$, we get
\begin{equation}\label{eq:2.18}
    \int\limits_0^1 (\widehat{f}_n(x-)-f(x))^2x\,dx\leq 2\sup_{0\leq x\leq 1} |\widehat{F}_n(x)-F(x)| \;\;\text{$(P$-a.s.$)$}.
\end{equation}
Next we apply the well-known Marshall's lemma which states that
\begin{equation}\label{eq:2.19}
    \sup_{0\leq x\leq 1} |\widehat{F}_n(x)-F(x)|\leq \sup_{0\leq x\leq 1} |F_n(x)-F(x)|
\end{equation}
and ultimately get the desired inequality
\begin{eqnarray}\label{eq:2.20}
    &\displaystyle \qquad \int\limits_0^1 (\widehat{f}_n(x-)-f(x))^2x\,dx\leq 2\sup_{0\leq x\leq 1} |F_n(x)-F(x)| \;\;\text{$(P$-a.s.$)$}. \qquad \qedhere
\end{eqnarray}
\end{proof}

The obtained inequality has several interesting implications. Firstly by the Glivenko--Cantelli theorem we know that with probability one
\begin{equation}\label{eq:2.21}
    \sup_{0\leq x\leq 1} |F_n(x)-F(x)|\longrightarrow 0\;\;\text{if}\;\; n\to\infty,
\end{equation}
hence we get the consistency of the Grenander estimator $\widehat{f}_n(x-)$, $0<x\leq 1$, in the sense of the weighted integrated square distance \eqref{eq:1.6}.

Next, consider the celebrated Dvoretzky--Kiefert--Wolfowitz inequality
\begin{equation}\label{eq:2.22}
    P\Big(\sqrt{n}\,\sup_{0\leq x\leq 1} |F_n(x)-F(x)|>\lambda\Big)\leq 2e^{-2\lambda^2} \;\;\text{for arbitrary}\;\; \lambda>0,
\end{equation}
where the sharp constant multiplier $2$ is due to Pascal Massart.

We get from inequalities \eqref{eq:2.20} and \eqref{eq:2.22}
\begin{multline}
    P\bigg(\sqrt{n}\int\limits_0^1 (\widehat{f}_n(x-)-f(x))^2x\,dx\leq \lambda\bigg) \\
    \geq P\Big(\sqrt{n}\,\sup_{0\leq x\leq 1} |F_n(x)-F(x)|\leq\frac{\lambda}{2}\Big)\geq 1-2\,e^{-\lambda^2/2}. \label{eq:2.23}
\end{multline}

Fix $\alpha$, $0<\alpha<1$, arbitrary small and take $\lambda=\sqrt{2\,\ln\frac{2}{\alpha}}$\,. Then we obtain from the latter inequality the following interesting estimate
\begin{equation}\label{eq:2.24}
    P\bigg(\int\limits_0^1 (\widehat{f}_n(x-)-f(x))^2x\,dx\leq \sqrt{2\,\ln\frac{2}{\alpha}}\cdot n^{-1/2}\bigg)\geq 1-\alpha.
\end{equation}
This estimate says that, for probability $1-\alpha$ close to $1$, we can find such a number $n$ of observations $X_1,X_2,\dots,X_n$, that the weighted integrated square error \eqref{eq:2.9} will be as small as required with the preseribed  probability $1-\alpha$.

From inequalities \eqref{eq:2.20} and \eqref{eq:2.22} we get also the following bound of the quadratic risk of the Grenander estimator
\begin{equation}\label{eq:2.25}
    E\int\limits_0^1 (\widehat{f}_n(x-)-f(x))^2x\,dx\leq \sqrt{2\pi}\cdot n^{-1/2}.
\end{equation}


\begin{thebibliography}{99}

\bibitem{1}
\textsc{Birg\'{e}, L.} (1987a).
Estimating a density under order restrictions: nonasymptotic minimax risk.
\textit{Ann. Statist.} \textbf{15} no. 3, 995--1012.
\MR{0902241}

\bibitem{2}
\textsc{Birg\'{e}, L.} (1987b).
On the risk of histograms for estimating decreasing densities.
\textit{Ann. Statist.} \textbf{15}, no. 3 1013--1022.
\MR{0902242}

\bibitem{3}
\textsc{Birg\'{e}, L.} (1989).
The Grenander estimator: a nonasymptotic approach.
\textit{Ann. Statist.} \textbf{17}, no. 4 1532--1549.
\MR{1026298}

\bibitem{4}
\textsc{Grenander, U.} (1956).
On the theory of mortality measurement. II.
\textit{Skand. Aktuarietidskr.} \textsc{39} 125--153.
\MR{0093415}

\bibitem{5}
\textsc{Grenander, U.} (1981).
\textit{Abstract inference}.
Wiley Series in Probability and Mathematical Statistics. John Wiley \& Sons, Inc., New York.
\MR{0599175}

\bibitem{6}
\textsc{Groeneboom, P.} (1985).
\textit{Estimating a monotone density}.
Proceedings of the Berkeley conference in honor of Jerzy Neyman and Jack Kiefer, Vol. II (Berkeley, Calif., 1983), 539--555, Wadsworth Statist./Probab. Ser., Wadsworth, Belmont, CA.
\MR{0822052}

\bibitem{7}
\textsc{Hewitt, E.} and \textsc{Stromberg, K.} (1975)
\textit{Real and abstract analysis. A modern treatment of the theory of functions of a real variable}, Third printing.
Graduate Texts in Mathematics, No. 25. Springer-Verlag, New York--Heidelberg.
\MR{0367121}

\bibitem{8}
\textsc{Kiefer, J.} and \textsc{Wolfowitz, J.} (1976).
Asymptotically minimax estimation of concave and convex distribution functions. Z.
\textit{Wahrscheinlichkeitstheorie und Verw. Gebiete} \textbf{34}, no. 1 73--85.
\MR{0397974}

\bibitem{9}
\textsc{Prakasa Rao, B. L. S.} (1969).
Estimation of a unimodal density. \textit{Sankhy\={a} Ser. A} \textbf{31} 23--36.
\MR{0267677}

\end{thebibliography}
\end{document}